\documentclass[a4j,12pt]{article}
\usepackage{amsmath,amsthm,amssymb,amscd,ascmac, amsfonts}
\usepackage{mathrsfs}
\usepackage{stmaryrd}
\usepackage{braket}
\usepackage{accents}
\usepackage{url}
\allowdisplaybreaks[4]

\usepackage[dvips]{graphicx,color,psfrag}

\makeatletter

\newtheorem*{multitheorem}{\variable@name}

\theoremstyle{definition}
\newcommand{\variable@name}{Theorem}
\newtheorem*{multiproclaim}{\variable@name}

\theoremstyle{plain}
\newtheorem{thm}{Theorem}

\newtheorem{cor}[thm]{Corollary}

\theoremstyle{definition}

\newtheorem{rem}[thm]{Remark}

\begin{document}
\title{An equivalent condition for the Markov triples and the Diophantine equation $a^{2}+b^{2}+c^{2}=abcf(a,b,c)$}
\author{Genki Shibukawa}
\date{
\small MSC classes\,:\,11D25, 11F20}
\pagestyle{plain}

\maketitle

\begin{abstract}
We propose an equivalent condition for the Markov triples, which was mentioned by H. Rademacher essentially.  
As an application, we mention the solvability of the Diophantine equation $a^{2}+b^{2}+c^{2}=abcf(a,b,c)$. 
\end{abstract}
Throughout the paper, we denote the set of non-negative integers by $\mathbb{Z}_{\geq 0}$, 
the ring of integers by $\mathbb{Z}$. 
Markov triples $(a,b,c)$ are positive integer solutions of the Markov Diophantine equation
\begin{equation}
\label{eq:the Markov equation}
a^{2}+b^{2}+c^{2}=3abc.
\end{equation}
If $(a,b,c)$ is a Markov triple, then $a,b,c$ are pairwise prime. 
Therefore, in this note, we assume that $a$, $b$, and $c$ are pairwise prime positive integers unless otherwise specified. 
The key condition as follows: 
\begin{equation}
\label{eq:Markov condition}
a^{2}+b^{2}\equiv 0 \quad (\mathrm{mod}\,c), \quad 
b^{2}+c^{2}\equiv 0 \quad (\mathrm{mod}\,a), \quad 
c^{2}+a^{2}\equiv 0 \quad (\mathrm{mod}\,b).
\end{equation}
\begin{thm}
\label{thm:an equivalent condition}
\quad $(\ref{eq:the Markov equation}) \quad \Leftrightarrow \quad (\ref{eq:Markov condition})$.
\end{thm}
\begin{proof}
$\Rightarrow )$ It is obvious. \\
$\Leftarrow  )$ Let $D(a;b,c)$ be the Dedekind (Rademacher) sums 
$$
D(a;b,c)
   :=\frac{1}{a}\sum_{k=1}^{a-1}\cot{\left(\frac{\pi bk}{a}\right)}\cot{\left(\frac{\pi ck}{a}\right)}
$$
and $s(a;b)$ be the usual Dedekind sum 
$$
s(a;b)
   :=
   \frac{1}{4}D(a;b,1)
   =
   \frac{1}{4}D(a;1,b).
$$
For $b^{\prime}$ that satisfies $bb^{\prime}\equiv 1\,\,(\mathrm{mod}\,a)$, we have
$$
D(a;b,c)
   =D(a;1,b^{\prime}c)
   =4s(a;b^{\prime}c). 
$$
Thus, by the zero condition for $s(a;b)$ \cite{2} p28
$$
b^{2}+1\equiv 0 \quad (\mathrm{mod}\,a) \quad \Leftrightarrow \quad s(a;b)=0,
$$
we obtain the zero condition for $D(a;b,c)$, that is
\begin{equation}
\label{eq:zero condition of D}
b^{2}+c^{2}\equiv 0 \quad (\mathrm{mod}\,a) \quad \Leftrightarrow \quad D(a;b,c)=0.
\end{equation}
From the assumption (\ref{eq:Markov condition}) and zero condition (\ref{eq:zero condition of D}), we derive
$$
D(a;b,c)=D(b;c,a)=D(c;a,b)=0.
$$
Finally, by the reciprocity law of $D(a;b,c)$ \cite{1}
$$
D(a;b,c)+D(b;c,a)+D(c;a,b)=\frac{a^{2}+b^{2}+c^{2}}{3abc}-1,
$$
the triple $(a,b,c)$ satisfies the Markov Diophantine equation (\ref{eq:the Markov equation}). 
\end{proof}
For any polynomial $f(a,b,c) \in \mathbb{Z}[a,b,c]$, we consider the Diophantine equation
\begin{equation}
\label{eq:the Diophantine equation}
a^{2}+b^{2}+c^{2}=abcf(a,b,c).
\end{equation}
\begin{thm}
\label{thm:unsolve thm}
The pairwise prime positive integer triple $(a,b,c)$ is a solution of the Diophantine equation (\ref{eq:the Diophantine equation}) if and only if $(a,b,c)$ is a Markov triple and satisfies $f(a,b,c)=3$.
\end{thm}
\begin{proof}
If the Diophantine equation (\ref{eq:the Diophantine equation}) has a pairwise prime positive integer solution $(a,b,c)$, then $(a,b,c)$ satisfies (\ref{eq:Markov condition}). 
Therefore, from Theorem \ref{thm:an equivalent condition}, we have
$$
abcf(a,b,c)=a^{2}+b^{2}+c^{2}=3abc.
$$
By $abc\not=0$, we have $f(a,b,c)=3$. 

On the other hand, if there exists a Markov triple $(a,b,c)$ such that $f(a,b,c)=3$, then  
$$
a^{2}+b^{2}+c^{2}=3abc=abcf(a,b,c).
$$
\end{proof}
\begin{rem}
Theorem \ref{thm:an equivalent condition} was mentioned by Rademacher \cite{1} Part III Lecture 32 essentially. 
However, it seems that no one mention Theorem \ref{thm:unsolve thm}. 
\end{rem}
From Theorem \ref{thm:unsolve thm}, we obtain the following results. 
\begin{cor}
Let $k$ be the greatest common divisor of $a$, $b$, and $c$;
$$
k:=\gcd(a,b,c).
$$
{\rm{(1)}} If $k$ is not $1$ or $3$,then the Diophantine equation (\ref{eq:the Diophantine equation}) has no positive integer solutions.\\
{\rm{(2)}} If $k=1$,then for any non-zero polynomial $g(a,b,c) \in \mathbb{Z}_{\geq 0}[a,b,c]$ the Diophantine equation 
\begin{equation}
\label{eq:modifeid g1}
a^{2}+b^{2}+c^{2}=abc(3+g(a,b,c))
\end{equation}
has no positive integer solutions.\\
{\rm{(3)}} If $k=3$,then for any polynomial $g(a,b,c)\not\equiv 1 \in \mathbb{Z}_{\geq 0}[a,b,c]$ the Diophantine equation 
\begin{equation}
\label{eq:modifeid g3}
a^{2}+b^{2}+c^{2}=abcg(a,b,c)
\end{equation}
has no positive integer solutions.
\end{cor}
\begin{proof}
We point out if coprime positive integers $a,b$ and $c$ satisfy the Diophantine equation (\ref{eq:the Diophantine equation}) then $a,b,c$ are pairwise prime. 
In fact, since 
$$
a\equiv 0 \,(\mathrm{mod}\,\gcd(b,c)),\quad
b\equiv 0 \,(\mathrm{mod}\,\gcd(c,a)),\quad
c\equiv 0\, (\mathrm{mod}\,\gcd(a,b)),
$$
we have 
$$
\gcd(b,c)=\gcd(c,a)=\gcd(a,b)=\gcd(a,b,c)=1.
$$
Therefore, if $(a,b,c)$ is a positive integer solution of (\ref{eq:the Diophantine equation}), then there exist pairwise prime integers $a_{0}$, $b_{0}$ and $c_{0}$ such that 
\begin{equation}
\label{eq:reduction of solution}
a=ka_{0},\,\,b=kb_{0},\,\,c=kc_{0}.
\end{equation}
{\rm{(1)}} By substituting (\ref{eq:reduction of solution}) to (\ref{eq:the Diophantine equation}), we have 
\begin{equation}
\label{eq:k reduction}
a_{0}^{2}+b_{0}^{2}+c_{0}^{2}
   =
   a_{0}b_{0}c_{0}kf(ka_{0},kb_{0},kc_{0}).
\end{equation}
From Theorem \ref{thm:unsolve thm}, if $(a,b,c)$ is a positive integer solution then 
\begin{equation}
\label{eq:deform eq}
kf(ka_{0},kb_{0},kc_{0})=3.
\end{equation}
Since $f(ka_{0},kb_{0},kc_{0})$ is a integer, $k=1$ or $3$ for the equation (\ref{eq:deform eq}) to have a solution.\\
{\rm{(2)}} For any non-zero polynomial $g(a,b,c)$ and positive integers $(a,b,c)$, $3+g(a,b,c)$ is greater than 3.     
Hence, from Theorem \ref{thm:unsolve thm}, (\ref{eq:modifeid g1}) has no integer solution.\\
{\rm{(3)}} By substituting $k=3$ to (\ref{eq:k reduction})
$$
a_{0}^{2}+b_{0}^{2}+c_{0}^{2}
   =
   a_{0}b_{0}c_{0}3f(3a_{0},3b_{0},3c_{0})
$$
and applying Theorem \ref{thm:unsolve thm}, we have 
$$
a_{0}^{2}+b_{0}^{2}+c_{0}^{2}
   =
   a_{0}b_{0}c_{0}3f(3a_{0},3b_{0},3c_{0}) \quad \Leftrightarrow \quad f(3a_{0},3b_{0},3c_{0})=1.
$$
On the other hand for any polynomial $g(a,b,c)$ with non-negative integer coefficients, $g(3a_{0},3b_{0},3c_{0})=1$ for positive integers $a_{0}$, $b_{0}$ and $c_{0}$ if and only if $g(a,b,c)\equiv 1$. 
\end{proof}


\bibliographystyle{amsplain}

\noindent 
Department of Mathematics, Graduate School of Science, Kobe University, \\
1-1, Rokkodai, Nada-ku, Kobe, 657-8501, JAPAN\\
E-mail: g-shibukawa@math.kobe-u.ac.jp

\end{document}